\documentclass{amsart}
\usepackage{mathrsfs,amssymb,mathrsfs, multirow,setspace}
\usepackage[a4paper, total={7in, 9in}]{geometry}
\usepackage{enumerate}
\theoremstyle{plain}
\usepackage{relsize}
\usepackage{graphicx}
\newtheorem{theorem}{Theorem}[section]
\newtheorem{lemma}[theorem]{Lemma}
\newtheorem{proposition}[theorem]{Proposition}
\newtheorem{corollary}[theorem]{Corollary}

\theoremstyle{definition}

\newtheorem{example}[theorem]{Example}

\theoremstyle{remark}

\input xy
\xyoption{all}

\begin{document}

	\title[On the idempotent graph of a ring]{On the idempotent graph of a ring}
\author[Praveen Mathil, Barkha Baloda, Jitender Kumar]{Praveen Mathil, Barkha Baloda, Jitender Kumar*}
	\address{Department of Mathematics, Birla Institute of Technology and Science Pilani, Pilani-333031, India}
	\email{maithilpraveen@gmail.com, barkha0026@gmail.com, jitenderarora09@gmail.com}

\begin{abstract}
Let $R$ be a ring with unity. The \emph{idempotent graph} $G_{\text{Id}}(R)$ of a ring $R$ is an undirected simple graph whose vertices are the set of all the elements of ring $R$ and two vertices $x$ and $y$  are adjacent if and only if $x+y$ is an idempotent element of $R$. In this paper, we obtain a necessary and sufficient condition on the ring $R$ such that $G_{\text{Id}}(R)$ is planar. We prove that $G_{\text{Id}}(R)$ cannot be an outerplanar graph. Moreover, we classify all the finite non-local commutative rings $R$ such that $G_{\text{Id}}(R)$ is a cograph, split graph and threshold graph, respectively. We conclude that latter two graph classes of $G_{\text{Id}}(R)$ are equivalent if and only if $R \cong \mathbb{Z}_2 \times \mathbb{Z}_2 \times \cdots \times \mathbb{Z}_2$.
\end{abstract}

	
\keywords{Non-local ring, idempotent elements, planar graph, forbidden graph classes\\ *  Corresponding author}
	
\maketitle

\section{Introduction}
Attaching a graph to algebraic structures, viz. groups, rings, vector spaces, and study their properties using graph theoretic parameters, became an interesting research area in the last thirty years. The study of these algebraic graphs leads to an interplay between properties of respective algebraic structure and graph theoretic properties of graphs associated with them. In 1988, Beck \cite{beck1988coloring} introduced the idea of associating a graph to an algebraic structure by giving the notion of the zero-divisor graphs of rings. Later on, Anderson and Livingston \cite{anderson1999zero} modified the definition of the zero-divisor graph of a commutative ring. Many authors defined and studied the graphs associated to ring structures viz. zero-divisor graph, total graph, unit graph, co-maximal graph, annihilating-ideal graph etc. (see \cite{afkhami2011cozero, akbari2009total, anderson2008zero, anderson2008total, anderson1999zero, ashrafi2010unit, behboodi2011annihilating, biswas2022subgraph, maimani2008comaximal}). Numerous authors have investigated the embedding of aforementioned algebraic graphs such that they can be drawn on a plane without edge crossing. The planarity of the zero-divisor graphs has been investigated in \cite{akbari2003zero,belshoff2007planar}. For more work related to the planarity of various graphs associated to rings, we refer the reader to \cite{belsi2021genus,jafari2010planarity, kavitha2017genus,maimani2012rings,su2015rings} and references therein.

Idempotent elements of a ring play an important role in the study of ring structure. To investigate the ring properties through their zero-divisors, idempotent elements, nilpotent elements, regular elements etc., many researchers have introduced and studied the algebraic graphs with respect to these elements (see \cite{akbari2013idempotent,anderson2012neumann,anderson1999zero,basnet2021nilpotent,cvetko2008indecomposability,razaghi2021graph}). Razzaghi \emph{et al.} \cite{razaghi2021graph} introduced the idempotent graph of a ring. They have studied certain graph theoretic properties including connectedness, girth and the diameter of the idempotent graph of a ring. The \emph{idempotent graph} $G_{\text{Id}}(R)$ of a ring $R$ is an undirected simple graph with vertex set $R$ and two vertices $x$ and $y$  are adjacent if and only if $x+y$ is an idempotent element of $R$.

In this paper, we obtain a necessary and sufficient condition on a finite non-local commutative ring $R$ such that the graph $G_{\text{Id}}(R)$ is planar. We also prove that $G_{\text{Id}}(R)$ cannot be an outerplanar graph. Moreover, we classify all the finite non-local commutative rings $R$ such that $G_{\text{Id}}(R)$ is a split graph, threshold graph and cograph, respectively. The paper is arranged as follows: Section 2 comprises basic definitions and relevant results from graph theory and ring theory. In Section 3, we discuss the planarity and outerplanarity of $G_{\text{Id}}(R)$. In Section 4, we determine all the finite non-local commutative rings $R$ for which  $G_{\text{Id}}(R)$ lies in some forbidden graph class, viz. split graph, threshold graph and cograph.

\section{Preliminaries}

In this section, we recall  necessary definitions, results and notations of graph theory from \cite{westgraph}.
A graph $\Gamma$ is a pair $ \Gamma = (V, E)$, where $V = V(\Gamma)$ and $E = E(\Gamma)$ are the set of vertices and edges of $\Gamma$, respectively.  A \emph{subgraph} $\Gamma'$ of $\Gamma$ is the graph such that $V(\Gamma') \subseteq V(\Gamma)$ and $E(\Gamma') \subseteq E(\Gamma)$. If $X \subseteq V(\Gamma)$, then the subgraph of $\Gamma$ induced by $X$, denoted by $\Gamma(X)$, is the graph with vertex set $X$ and two vertices of $\Gamma(X)$ are adjacent if and only if they are adjacent in $\Gamma$. Two distinct vertices $x, y \in \Gamma$ are $\mathit{adjacent}$, denoted by $x \sim y$ or $(x,y)$, if there is an edge between $x$ and $y$. Otherwise, we write it as $x \nsim y$. Let  $\Gamma_1$ and  $\Gamma_2$ be two graphs. The \emph{union} $\Gamma_1 \cup \Gamma_2$ of the graphs $\Gamma_1$ and $\Gamma_2$ is the graph with $V(\Gamma_1 \cup \Gamma_2) = V(\Gamma_1) \cup V(\Gamma_2)$ and $E(\Gamma_1 \cup \Gamma_2) = E(\Gamma_1) \cup E(\Gamma_2)$. A graph $\Gamma$ is said to be \emph{complete} if every two distinct vertices are adjacent. The complete graph on $n$ vertices is denoted by $K_n$. A \emph{path} $P_n$ in a graph is a sequence of $n$ distinct vertices with the property that each vertex in the sequence is adjacent to the next vertex of it.
A \emph{cycle} is a path that begins and ends on the same vertex. A cycle of length $n$ is denoted by $C_n$. A graph $\Gamma$ is \emph{bipartite}  if $V(\Gamma)$ is the union of two disjoint independent  sets. A graph $\Gamma$ is said to be a \emph{complete bipartite graph}, denoted by $K_{m,n}$, if the vertex set $V(\Gamma)$ can be partitioned into two nonempty sets $A$ of $m$ vertices and $B$ of $n$ vertices such that two distinct vertices are adjacent if and only if they do not belong to the same set. The \emph{subdivision} of the edge $(u,v)$ in a graph $\Gamma$ is the deletion of the edge $(u,v)$ from $\Gamma$ and the addition of two edges $(u,w)$ and $(w,v)$ along with a new vertex $w$. A graph obtained from $\Gamma$ by a sequence of edge subdivisions is called a subdivision of $\Gamma$. \emph{Contraction} of the edge $(u,v)$ of $\Gamma$, denoted by the vertex $[u,v]$, is an operation that deletes the edge $(u,v)$ from $\Gamma$ and merges the vertices $u$ and $v$ in $\Gamma$. If $u_1 \sim u_2 \sim \cdots \sim u_k$, where $k \ge 2$, then the vertex obtained by the contraction of the edges $(u_1, u_2)$, $(u_2, u_3)$, \ldots, $(u_{k-1}, u_k)$ is denoted by $[u_1, u_2, \ldots, u_k]$. An undirected graph $\Gamma'$ is a minor of another undirected graph $\Gamma$ if $\Gamma'$ can be obtained from $\Gamma$ by contracting or deleting some edges. Two graphs are said to be \emph{homeomorphic} if both can be obtained from the same graph by subdivision or contracting of edges.
 
A graph $\Gamma$ is \emph{outerplanar} if it can be embedded in the plane such that all vertices lie on the outer face.  A graph $\Gamma$ is \emph{planar} if it can be drawn on a plane without edge crossing. It is well known that every outerplanar graph is a planar graph. Now we have the following known results related to outerplanar and  planar graphs.

\begin{theorem}\cite{westgraph}\label{outerplanar criteria}
A graph $\Gamma$ is outerplanar if and only if it  does not contain a subdivision of $K_4$ or $K_{2,3}$.
\end{theorem}
\begin{theorem}\cite{westgraph}\label{planar criteria}
A graph $\Gamma$ is planar if and only if it does not contain a subdivision of $K_5$ or $K_{3,3}$.
\end{theorem}

A graph $\Gamma$ is a \emph{split} graph if the vertex set is the disjoint union of  two sets $A$ and $B$, where $A$ induces a complete subgraph and $B$ is an independent set $\Gamma$.

\begin{lemma}\label{splitgraph}\cite{MR0505860}
A graph $\Gamma$ is a split graph if and only if it does not have an induced subgraph isomorphic to one of the three forbidden graphs, $C_4, C_5$ or $2K_2$.
\end{lemma}

A graph $\Gamma$ is said to be a \emph{cograph} if it has no induced subgraph isomorphic to $P_4$. A \emph{threshold} graph is the graph which does not contain an induced subgraph isomorphic to $P_4, C_4$ or $2K_2$, where $2K_2$ denotes the two copies of the complete graph $K_2$. Every threshold  graph is a cograph as well as a split graph. A \emph{cactus graph} is a connected graph where any two simple cycles have at most one vertex in common. A connected graph is said to be \emph{unicyclic} if it contains exactly one cycle. 

Throughout the paper, $R$ is a finite non-local commutative ring with unity. For basic definitions of ring theory, we refer the reader to \cite{atiyah1994introduction}. For $x \in R$, $\langle x \rangle$ denotes the ideal generated by $x$. An  element $x$ of a ring $R$ is said to be an \emph{idempotent} element if $x^2 = x$. The set of idempotent elements of ring $R$ is denoted by $\text{Id}(R)$. The \emph{characteristic} of a ring $R$ with unity $1$, denoted by $char(R)$, is the smallest positive integer $n$ such that $n \cdot 1 = 0$. Let $R$ be a finite non-local commutative ring. 
 By the structural theorem (see \cite{atiyah1994introduction}),  $R$ is uniquely (up to isomorphism) a finite direct product of  local rings $R_i$ that is $R \cong R_1 \times R_2 \times \cdots \times R_n$, where $n \geq 2$.

The following results on $G_{\textnormal{Id}}(R)$ are useful in the sequel.

\begin{proposition}{\cite[Proposition 2.2]{razaghi2021graph}}\label{idempotent_degree}
Let $R$ be a ring, $x \in R$ and $\textnormal{Id}(R)$ be finite. Then the following statements hold for $G_{\textnormal{Id}}(R)$:
\begin{enumerate}[\rm(i)]
    \item If $2x \in \textnormal{Id}(R)$, we have $deg (x) = \textnormal{Id}(R)-1$; and

    \item If $2x \not\in \textnormal{Id}(R)$, we have $deg (x) = \textnormal{Id}(R)$.
\end{enumerate}
 \end{proposition}

 \begin{theorem}{\cite[Theorem 3.1]{razaghi2021graph}}\label{idempotent_connected}
  Let $R$ be a ring. Then $G_{\textnormal{Id} }(R)$is connected if and only if $(R,+)$ = $\langle \textnormal{Id}(R) \rangle$ (R is additively generated by its idempotents).   
 \end{theorem}
 
\begin{corollary}{\cite[Corollary 3.3]{razaghi2021graph}}\label{Idempotent_path}
Let $R$ be a ring. Then $(R,+) =\  \langle \textnormal{Id}(R) \rangle$ and $R$ has no nontrivial idempotents if and only if idempotent graph $G_{\textnormal{Id}}(R)$ is a path graph.
\end{corollary}

\begin{theorem}{\cite[Theorem 3.7]{razaghi2021graph}}\label{idempotent_subringpath}
Let $R$ be a ring with no nontrivial idempotents. If $G_{\textnormal{Id}}(R)$ is disconnected, then there is a subring $S$ of $R$ such that $G_{\textnormal{Id}}(S)$ is a path graph. Set
$A = \{x + S| \ x \in R \setminus S \ \text{and}\  2x = 0\}$ and $B = \{x + S | \ x \in R \ \text{and} \ 2x \neq 0 \}$. Then
$G_{\textnormal{Id}}(R)$ is disjoint union of $|A| + 1$ times $G_{\textnormal{Id}}(S)$ and $|B|/2$ times bipartite subgraph of $G_{\textnormal{Id}}(R)$. In particular, if $2x \notin \textnormal{Id}(R)$, then $G_{\textnormal{Id}}(R)$ is disjoint union of $|A| + 1$ times
$G_{\textnormal{Id}}(S)$ and $|B|/2$ times 2-regular bipartite subgraph of $G_{\textnormal{Id}}(R)$.
\end{theorem}

\section{Planarity of $G_{\textnormal{Id}}(R)$}

In this section, we investigate the embedding of $G_{\textnormal{Id}}(R)$ on a plane without edge crossings. We prove that for a non-local commutative ring $R$ the graph $G_{\textnormal{Id}}(R)$ cannot be outerplanar (see Theorem \ref{outerplanarity}). Also, we classify all non-local commutative rings $R$ such that the graph $G_{\textnormal{Id}}(R)$ is planar (see Theorem \ref{planarity}).

\begin{theorem}\label{outerplanarity}
 For a non-local commutative ring $R$, the graph $G_{\textnormal{Id}}(R)$ cannot be outerplanar.   
\end{theorem}

  \begin{proof}
Let $R$ be a non-local commutative ring. Then $R \cong R_1 \times R_2 \times  \cdots  \times R_n (n \ge 2)$, where each $R_i$ is a local ring. Suppose $G_{\textnormal{Id}}(R)$ is an outerplanar graph. First suppose that $n \geq 3$ and for $i = 1, 2$, assume that $char(R_i) \notin \{2, 3\}$. Then there exist $a_1 \in R_1 \setminus \text{Id}(R_1)$, $a_2 \in R_2 \setminus \text{Id}(R_2)$ such that $2a_1 \notin R_1$ and $2a_2 \notin R_2$. Consider the sets $X =  \{ (a_1, a_2, 1, 0, \ldots, 0), (a_1, a_2, 0, \ldots, 0) \}$ and  $ Y = \{ (-a_1, -a_2, 0, \ldots, 0), (-a_1, -a_2, 1, 0, \ldots, 0), (1-a_1, -a_1, 0, \ldots, 0) \}$. Note that the subgraph $G_{\textnormal{Id}}(R)$ induced by the set $X \cup Y$ is isomorphic to $K_{2, 3}$, a contradiction. Next, assume that the characteristic of one of the local ring $R_i$ is two. Without loss of generality, assume that $char(R_1) = 2$. It follows that the set $\{(0, 0, \ldots, 0), (1, 0, \ldots, 0), (1, 1, \ldots, 0), (0, 0, 1, 0, \ldots, 0)\}$ induces a subgraph isomorphic to $K_4$, which is not possible. If $char(R_1) = 3$, then by Theorem \ref{idempotent_subringpath}, there exists an induced path $P: a_1 \sim a_2 \sim a_3$ in $G_{\textnormal{Id}}(R_1)$. Note that $2a_1, 2a_3 \in \textnormal{Id}(R_1)$. It follows that the set $\{(a_1, 0, 0, \ldots, 0), (a_1, 1, 0, \ldots, 0), (a_2, 0, \ldots, 0), (a_1, 0, 1, 0, \ldots, 0)\}$ induces $K_4$ as a subgraph of $G_{\textnormal{Id}}(R)$, a contradiction. If $char(R_1) = 2$, then note that the set $X = \{(1, 0, 0, \ldots, 0)$, $(0, 0, \ldots, 0)$, $(1, 1, 0, \ldots, 0)$, $(1, 0, 1, 0, \ldots, 0)\}$ induces a subgraph isomorphic to $K_4$, which is not possible. It follows that $n = 2$. Thus, $R \cong R_1 \times R_2$. 

First suppose that $char(R_i) \geq 4$ for each $i \in \{ 1,2\}$. By Theorem \ref{idempotent_subringpath}, there exists an induced path $a_1 \sim a_2 \sim \sim a_3 \sim a_4$ in $G_{\textnormal{Id}}(R_1)$ and $b_1 \sim b_2 \sim b_3 \sim b_4$ in $G_{\textnormal{Id}}(R_2)$. It follows that the subgraph induced by the set $\{(a_1, b_1), (a_1, b_3), (a_2, b_2), (a_1, b_2), [(a_2, b_1),(a_1,b_3), (a_2,b_4)]\}$ is isomorphic to $K_{2,3}$, which is not possible. Consequently, either $char(R_1) \leq 3$ or $char(R_2) \leq 3$. Without loss of generality, assume that $char(R_1) = 3$ and $char(R_2) \geq 4$. By Theorem \ref{idempotent_subringpath}, there exists an induced path $x_1 \sim x_2 \sim x_3$ in $G_{\textnormal{Id}}(R_1)$ and $y_1 \sim y_2 \sim y_3 \sim y_4$ in $G_{\textnormal{Id}}(R_2)$. Then the set $X =\{(x_1, y_1), (x_1, y_3), (x_2, y_2), (x_1, y_2), [(x_2, y_1),(x_1,y_3), (x_2,y_4)]\}$ induces $K_{2,3}$ as a subgraph of $G_{\textnormal{Id}}(R)$, a contradiction. Similarly, if $char(R_1)= char(R_2) = 3$, then we get $K_{2,3}$ as an induced subgraph of $G_{\textnormal{Id}}(R)$, again a contradiction. Let $char(R_i) = 2$ for some $i$. If $char(R_1) = 2$ and $char(R_2) =k(k\ge2)$, there exists an induced path $b_1 \sim b_2 \sim \cdots \sim b_k$ in $G_{\textnormal{Id}}(R_2)$. Then the set $Y = \{(0, b_1), (1, b_1), (1, b_2), [(0, b_2),(0,b_3)]\}$ induces $K_4$ as a subgraph of $G_{\textnormal{Id}}(R)$. This completes our proof.
  \end{proof}  

In view of the proof of the Theorem \ref{outerplanarity}, we have the following proposition.

  \begin{proposition}\label{cactusgraph}
  For a non-local commutative ring $R$,
    \begin{enumerate}
 \item[{\rm(i)}] $G_{\textnormal{Id}}(R)$ cannot be a cactus graph.
 \item[{\rm(ii)}] $G_{\textnormal{Id}}(R)$ cannot be a unicyclic graph.
 \end{enumerate}
 \end{proposition}

  \begin{theorem}\label{planarity}
 Let $R \cong R_1 \times R_2 \times  \cdots  \times R_n (n \ge 2)$ be a non-local commutative ring. Then $G_{\textnormal{Id}}(R)$ is a planar graph if and only if $R \cong R_1 \times R_2$ and one of the following holds:
 \begin{enumerate}
 \item[{\rm(i)}] $(R_1, +) = \langle \textnormal{Id}(R_1) \rangle$ and $(R_2, +) = \langle \textnormal{Id}(R_2) \rangle$.
 \item[{\rm(ii)}]  $(R_1, +) = \langle \textnormal{Id}(R_1) \rangle$ and $char(R_2)= 2$.
 \item[{\rm(iii)}] $char(R_1)= 2$ and  $char(R_2)= 2$.
 \end{enumerate}
\end{theorem}

  \begin{proof}
 First, suppose that the graph $G_{\textnormal{Id}}(R)$ is planar. We prove the result in the following cases.

 \textbf{Case-1:} $n \ge 4$. Note that the set $X = \{(0, 0, \ldots, 0)$, $(1, 0, \ldots, 0)$, $(0, 1,0, \ldots, 0)$, $(0, 0, 1, 0, \ldots, 0)$, $(0, 0, 0, 1, 0, \ldots, 0) \}$ induces $K_5$ as a subgraph of $G_{\textnormal{Id}}(R)$. Therefore, the graph $G_{\textnormal{Id}}(R)$ is not planar. 

 \textbf{Case-2:} $n = 3$ i.e. $R \cong R_1 \times R_2 \times R_3$. Without loss of generality, assume that $char(R_1) = char(R_2) = 2$. Then there exist $a_1, a_2 \in R_1 $, $b_1, b_2 \in R_2$ such that $a_i + a_j \in \textnormal{Id}(R_1)$ and $b_i + b_j \in \textnormal{Id}(R_2)$ for each $i,j \in \{1,2\}$. Note that the set $\{(a_1, b_1, 0)$, $(a_1, b_2, 0)$, $(a_2, b_1, 0)$, $(a_2, b_2, 0)$, $(a_1, b_1, 1) \}$ induces $K_5$ as an induced subgraph of $G_{\textnormal{Id}}(R)$. Therefore, $char(R_i) \neq 2$ for some $i$. Let $char(R_1) = 2$ but $char(R_2) \neq 2$ and $char(R_3) \neq 2$. Then there exist $a_1, a_2 \in R_1$ such that $a_i + a_j \in \textnormal{Id}(R_1)$ for $i,j \in \{ 1,2\}$. Consider the set $X = \{(a_1, 1, 0)$, $(a_1, 1, 1)$, $(a_2, 1, 1)$, $(a_1, -1, 0)$, $(a_2, -1, 0)$, $(a_1, 0, 0)\}$. Then the subgraph induced by the set $X$ is isomorphic to $K_{3,3}$. Therefore, $char(R_i) \neq 2$ for each $i$. Assume that $char(R_i) \geq 3$ for each $i \in \{ 1,2,3\}$. Let $a_1 \in R_1 \setminus \textnormal{Id}(R_1)$ such that $2a_1 \notin \textnormal{Id}(R_1)$. Then $G_{\textnormal{Id}}(R)$ contains a subgrah  homeomorphic to $K_{3,3}$ (see Figure \ref{planar_1_idempotent}), a contradiction. Thus, the graph $G_{\textnormal{Id}}(R)$ is not planar. 

\begin{figure}[h!]
			\centering
			\includegraphics[width=0.5 \textwidth]{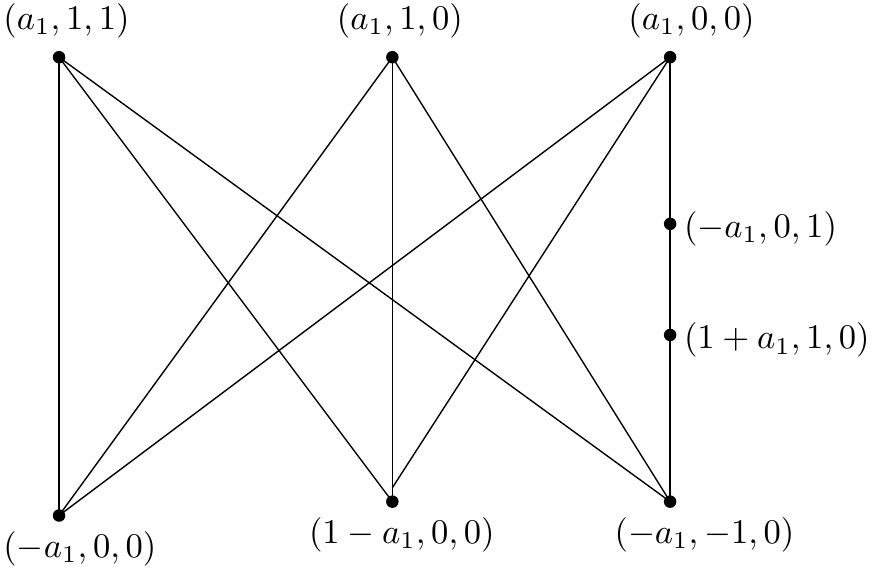}
			\caption{Subdivision of $G_{\textnormal{Id}}(R_1 \times R_2 \times R_3)$ homeomorphic to $K_{3,3}$}
   \label{planar_1_idempotent}
	\end{figure}

\textbf{Case-3:} $n = 2$ i.e. $R \cong R_1 \times R_2$. Suppose one of the local ring $R_i$ does not have characteristic $2$ and $(R_i, +) \neq \langle \textnormal{Id}(R_i) \rangle$. Without loss of generality, let  $(R_1, +) \neq \langle \textnormal{Id}(R_1) \rangle$ and $char(R_1) \neq 2$. Then by Theorem \ref{idempotent_subringpath}, $G_{\textnormal{Id}}(R_1)$ has an induced cycle $C : x_1 \sim x_2 \sim \cdots \sim x_k \sim x_1$$(k \ge 6)$ of even length. Since the local ring $R_2$ has no non-trivial idempotent element, $G_{\textnormal{Id}}(R_2)$ has an induced path $ P: y_1 \sim y_2 \sim \cdots \sim y_r \sim y_1$$(r \ge 2)$ as a subgraph. Consider the set $A = \{ u_1, u_2, u_3, u_4,u_5 \}$. If $r$ is even, then consider
\begin{align*}
u_1 &= [(x_1,y_1),(x_2,y_1), \ldots, (x_{\frac{k}{2}},y_1)], \ \text{the contraction of the edges} \ ((x_1,y_1),(x_2,y_1)), \cdots, ((x_{\frac{k}{2}-1},y_1),(x_{\frac{k}{2}},y_1)) \\
u_2 &= [(x_{\frac{k}{2} +1},y_1),(x_{\frac{k}{2}+2},y_1), \ldots, (x_k,y_1)], \\
u_3 &= [(x_1,y_2), (x_2,y_3), (x_1,y_4),\ldots,(x_1,y_r),(x_2,y_r),(x_3,y_r),\ldots, (x_{\frac{k}{2}},y_r)], \\
u_4 &=[(x_k,y_2), (x_{k-1},y_3), (x_k,y_4),\ldots,(x_k,y_r),(x_{k-1},y_r),(x_{k-2},y_r),\ldots, (x_{\frac{k}{2} +1},y_r)] \ \text{and} \\
u_5 &= [(x_{\frac{k}{2}}, y_{r-1}), (x_{\frac{k}{2}+1}, y_{r-2}),(x_{\frac{k}{2}}, y_{r-3}),\ldots, (x_{\frac{k}{2}}, y_2)].    
\end{align*}
If $r$ is odd, then consider
\begin{align*}
u_1 &= [(x_1,y_1),(x_2,y_1), \ldots, (x_{\frac{k}{2}},y_1)],\\
u_2 &= [(x_{\frac{k}{2} +1},y_1),(x_{\frac{k}{2}+2},y_1), \ldots, (x_k,y_1)], \\
u_3 &= [(x_1,y_2), (x_2,y_3), (x_1,y_4),\ldots,(x_2,y_r),(x_3,y_r),(x_4,y_r),\ldots, (x_{\frac{k}{2}},y_r)], \\
u_4 &= [(x_k,y_2), (x_{k-1},y_3), (x_k,y_4),\ldots,(x_{k-1},y_r),(x_{k-2},y_r),(x_{k-3},y_r),\ldots, (x_{\frac{k}{2} +1},y_r)] \ \text{and} \\
u_5 &= [(x_{\frac{k}{2}}, y_{r-1}), (x_{\frac{k}{2}+1}, y_{r-2}),(x_{\frac{k}{2}}, y_{r-3}),\ldots, (x_{\frac{k}{2}+1}, y_2)].  
\end{align*}
 Then the subgraph (see Figure \ref{K_5idempotent}) induced by the set $A$ is homeomorphic to $K_5$, a contradiction.
 \begin{figure}[h!]
			\centering
			\includegraphics[width=0.5 \textwidth]{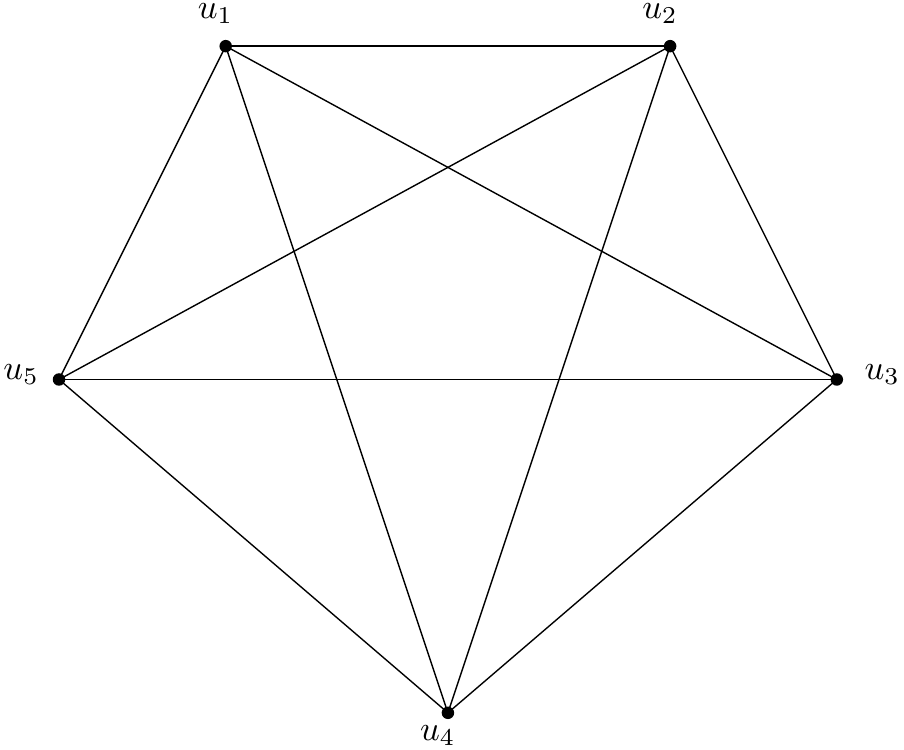}
			\caption{Subgraph of $G_{\textnormal{Id}}(R_1 \times R_2)$ homeomorphic to $K_5$}
   \label{K_5idempotent}
\end{figure}
Thus, for $G_{\textnormal{Id}}(R_1 \times R_2)$ to be planar, $R\cong R_1 \times R_2$ must satisfies one of the conditions given in the statement.

To prove the converse, let $R \cong R_1 \times R_2$ and it satisfies one of the given conditions. Then by Theorem \ref{idempotent_subringpath}, $G_{\textnormal{Id}}(R_1)$ and $G_{\textnormal{Id}}(R_2)$ are the  union of disjoint path graphs. Consider the paths $P_i$ and $P_j$ as a connected components of $G_{\textnormal{Id}}(R_1)$ and $G_{\textnormal{Id}}(R_2)$, respectively. Note that, all the connected components of the graph $G_{\textnormal{Id}}(R_1 \times R_2)$ are induced by the vertex set $V(P_i) \times V(P_j)$. Let $P_1 : u_1 \sim u_2 \sim \cdots \sim u_k  (k \ge 2)$ and $P_2 : v_1 \sim v_2 \sim \cdots \sim v_r  (r \ge 2)$ be two disjoint connected component of $G_{\textnormal{Id}}(R_1)$ and $G_{\textnormal{Id}}(R_2)$, respectively. Consider the subset $X = \{ (u_i, v_i) ; \ 1 \le i \le k, \ 1 \le j \le r \}$ of $R_1 \times R_2$. If both $k$ and $r$ are even, then $G_{\textnormal{Id}}(X)$ is planar (see Figure \ref{planar_idempotent}, where $i,j$ denotes the vertex $(u_i, v_j)$). Similarly, one can obtain a planar drawing of $G_{\textnormal{Id}}(X)$, for other cases of $k$ and $r$. Consequently, every connected component of $G_{\textnormal{Id}}(R)$ is planar and hence, the graph $G_{\textnormal{Id}}(R)$ is planar.
 \begin{figure}[h!]
			\centering
			\includegraphics[width=0.9 \textwidth]{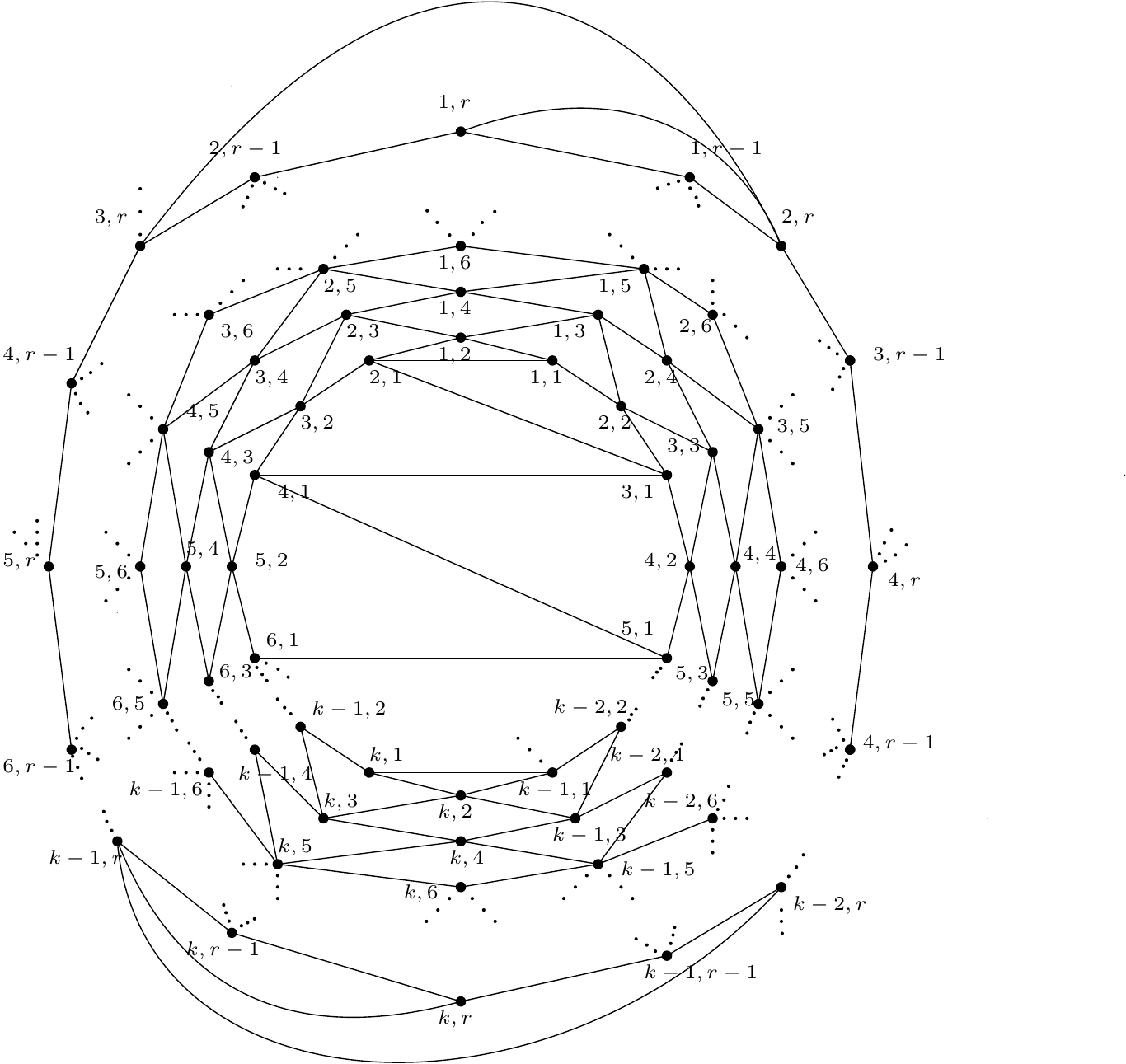}
			\caption{Planar drawing of $G_{\textnormal{Id}}(R_1 \times R_2)$. }
   \label{planar_idempotent}
\end{figure}
\end{proof}

The following examples illustrate that one of the conditions given in Theorem \ref{planarity} is necessary for the planarity of $G_{\textnormal{Id}}(R)$. 

\begin{example}
Consider $R \cong R_1 \times R_2$ such that $R_1 = \frac{\mathbb{Z}_3[x]}{\langle x^2 \rangle }$ and $R_2 = \mathbb{Z}_2$. Then note that (i) $char(R_1) \neq 2$ and $(R_1, +) \neq \langle \text{Id}(R_1) \rangle$, (ii) $(R_2, +) = \langle \text{Id}(R_2) \rangle$ and $char(R_2) = 2$. For the vertices $u_1 = x + \langle x^2 \rangle $, $u_2 = 2x + \langle x^2 \rangle $, $u_3 = x + 1 + \langle x^2 \rangle $, $u_4 = 2x +2  + \langle x^2 \rangle $, $u_5 = x + 2 + \langle x^2 \rangle $, $u_6 = 2x + 1 + \langle x^2 \rangle $ of $G_{\textnormal{Id}}(R_1)$, note that $u_1 \sim u_2 \sim \cdots \sim u_6 \sim u_1$ be the induced cycle of $G_{\textnormal{Id}}(R_1)$. Now consider the set $X = \{ x_1, x_2, x_3, x_4, x_5 \}$, where $x_1 = [(u_1,0), (u_2,0), (u_3,0)]$, $x_2= [(u_4,0), (u_5,0), (u_6,0)]$, $x_3 = (u_1,1)$, $x_4= (u_6, 1)$, $x_5 = [(u_2,1), (u_3,1), (u_4,1), (u_5,1)]$. Then the subgraph of $G_{\textnormal{Id}}(R_1 \times R_2)$ which is induced by the set $X$ is homeomorphic to $K_5$. Hence, $G_{\textnormal{Id}}\left(\frac{\mathbb{Z}_3[x]}{\langle x^2 \rangle } \times  \mathbb{Z}_2 \right)$ is not planar.
\end{example}

\begin{example}
Consider $R \cong R_1 \times R_2$ such that $R_1 = \frac{\mathbb{Z}_3[x]}{\langle x^2 \rangle }$ and $R_2 = \mathbb{Z}_3$. Then note that (i) $char(R_1) \neq 2$ and $(R_1, +) \neq \langle \text{Id}(R_1) \rangle$, (ii) $(R_2, +) = \langle \text{Id}(R_2) \rangle$ and $char(R_2) \neq 2$. For the vertices $u_1 = x + \langle x^2 \rangle $, $u_2 = 2x + \langle x^2 \rangle $, $u_3 = x + 1 + \langle x^2 \rangle $, $u_4 = 2x +2  + \langle x^2 \rangle $, $u_5 = x + 2 + \langle x^2 \rangle $, $u_6 = 2x + 1 + \langle x^2 \rangle $ of $G_{\textnormal{Id}}(R_1)$, note that $u_1 \sim u_2 \sim \cdots \sim u_6 \sim u_1$ be the induced cycle of $G_{\textnormal{Id}}(R_1)$. Consider the set $Y = \{ y_1, y_2, y_3, y_4, y_5\}$, where $y_1 = [(u_1,0), (u_2,0), (u_3,0)]$, $y_2= [(u_4,0), (u_5,0), (u_6,0)]$, $y_3 = [(u_1,1), (u_2,2), (u_3,2)]$, $y_4= [(u_6, 1), (u_5, 2), (u_4, 2)]$, $y_5 = (u_3,1)$. Then the subgraph of $G_{\textnormal{Id}}(R_1 \times R_2)$ induced by the set $Y$ is homeomorphic to $K_5$. Hence, $G_{\textnormal{Id}}\left(\frac{\mathbb{Z}_3[x]}{\langle x^2 \rangle } \times  \mathbb{Z}_3 \right)$ is not planar.
\end{example}

\begin{example}
Consider $R \cong R_1 \times R_2$ such that $R_1 = \frac{\mathbb{Z}_3[x]}{\langle x^2 \rangle }$ and $R_2 = \frac{\mathbb{Z}_3[x]}{\langle x^2 \rangle}$. Then note that (i) $char(R_1) \neq 2$ and $(R_1, +) \neq \langle \text{Id}(R_1) \rangle$, (ii) $(R_2, +) \neq \langle \text{Id}(R_2) \rangle$ and $char(R_2) \neq 2$. Let $u_1 = x + \langle x^2 \rangle $, $u_2 = 2x + \langle x^2 \rangle $, $u_3 = x + 1 + \langle x^2 \rangle $, $u_4 = 2x +2  + \langle x^2 \rangle $, $u_5 = x + 2 + \langle x^2 \rangle $, $u_6 = 2x + 1 + \langle x^2 \rangle $, $v_1 = 0 + \langle x^2 \rangle $, $v_2 = 1 + \langle x^2 \rangle $ and $v_3 = 2 + \langle x^2 \rangle $ be all the elements of $\frac{\mathbb{Z}_3[x]}{\langle x^2 \rangle}$. Consider the set $Z = \{ z_1, z_2, z_3, z_4, z_5\}$, where $z_1 = [(u_1,v_1), (u_2,v_1), (u_3,v_1)]$, $z_2= [(u_4,v_1), (u_5,v_1), (u_6,v_1)]$, $z_3 = [(u_1,v_2), (u_2,v_3), (u_3,v_3)]$, $z_4= [(u_6, v_2), (u_5, v_3), (u_4, v_3)]$, $z_5 = (u_3,v_1)$. Then the subgraph of $G_{\textnormal{Id}}(R_1 \times R_2)$ induced by the set $Z$ is homeomorphic to $K_5$. Hence, $G_{\textnormal{Id}}\left(\frac{\mathbb{Z}_3[x]}{\langle x^2 \rangle } \times \frac{\mathbb{Z}_3[x]}{\langle x^2 \rangle } \right)$ is not planar.
\end{example}

  
\section{Forbidden graph classes of $G_{\textnormal{Id}}(R)$}
In this section, we classify all the non-local commutative rings $R$ such that $G_{\textnormal{Id}}(R)$ is a split graph, threshold graph and cograph, respectively.  

\begin{theorem}\label{splitgraphtheorem}
Let $R$ be a non-local commutative ring.  Then $G_{\textnormal{Id}}(R)$ is a split graph if and only if $R \cong \mathbb{Z}_2 \times \mathbb{Z}_2 \times \cdots \times \mathbb{Z}_2$.
\end{theorem}

\begin{proof}
Let $R$ be a non-local commutative ring. Then $R \cong R_1 \times R_2 \times  \cdots  \times R_n$, where $n \geq 2$. Assume that the cardinality of one of the local ring $R_i$ is greater than $2$. Without loss of generality, let $|R_1| > 2$. Consider the sets $ X = \{(0,0,\ldots,0), (1, 0, \ldots,0)\}$ and $ Y = \{(a_1, 0, \ldots, 0), (-a_1, 0, \ldots, 0)\}$, where $a_1 \in R_1 \setminus \text{Id}(R_1)$. Then the subgraph of $G_{\textnormal{Id}}(R)$ induced by the set $X \cup Y$ is isomorphic to $2K_2$, a contradiction. 
It follows that $|R_i| = 2$ for each $i$, where $1 \le i \le n$. Therefore, $R \cong \mathbb{Z}_2 \times \mathbb{Z}_2 \times \cdots \times \mathbb{Z}_2$.
For the converse, if $R \cong \mathbb{Z}_2 \times \mathbb{Z}_2 \times \cdots \times \mathbb{Z}_2$, the note that $G_{\textnormal{Id}}(R) \cong K_{|R|}$. Hence, $G_{\textnormal{Id}}(R)$ is a split graph.
\end{proof}

\begin{corollary}
 Let $R$ be a non-local commutative ring. Then the following conditions are equivalent:
 \begin{enumerate}
 \item[{\rm(i)}] $G_{\textnormal{Id}}(R)$ is a split graph.
   \item[{\rm(ii)}] $G_{\textnormal{Id}}(R)$ is a threshold graph.
   \item[{\rm(iii)}] $R \cong \mathbb{Z}_2 \times \mathbb{Z}_2 \times \cdots \times \mathbb{Z}_2$.
 \end{enumerate} 
\end{corollary}

  \begin{theorem}
  Let $R \cong R_1 \times R_2 \times  \cdots  \times R_n$ ($n \geq 2$) be a non-local commutative ring. Then $G_{\textnormal{Id}}(R)$ is a cograph if and only if one of the following holds: 
  \begin{enumerate}
 \item[{\rm(i)}] $char(R_i) = 2$ for each $i$.
 \item[{\rm(ii)}] $R \cong \mathbb{Z}_3 \times R_2 \times R_3 \times \cdots \times R_n$ with $char(R_i) = 2$, where $2 \leq i \leq n$.
 \end{enumerate}
  \end{theorem} 
  
\begin{proof}
 First, suppose that $G_{\textnormal{Id}}(R)$ is a cograph. Let $n \geq 3$. Assume that $char(R_i) \neq 2$ for some $i$. Without loss of generality, assume that both $char(R_1)$ and $char(R_2)$ are not equal to $2$. Then note that the vertices $(0,0,\ldots,0)$, $(1, 0, \ldots,0)$, $(-1, 1, 0, \ldots, 0)$ and $ (1, -1, 0, \ldots, 0)$ induces a subgraph isomorphic to $P_4$, a contradiction. Now, suppose that the characteristic of one of the local ring $R_i$ is not two. Without loss of generality, let $char(R_1) \neq 2$. If $char(R_1) \neq 3$, then there exists $a_1 \in R_1$ such that $2a_1 \notin \textnormal{Id}(R_1)$. It follows that the subgraph induced by the set \{$(a_1, 1, 0, \ldots, 0),  (1-a_1,0, \ldots, 0), (a_1-1, 1, 1,0, \ldots, 0), (1-a_1, -1, -1, )$\} is isomorphic to $P_4$, a contradiction. Next, let $char(R_1) = 3$ but $R_1 \ncong \mathbb{Z}_3$. Then by Theorem \ref{idempotent_subringpath}, note that $G_{\textnormal{Id}}(R_1)$ contains at least one induced cycle of length at least six. Suppose that $x_1 \sim x_2 \sim x_3 \sim x_4 \sim x_5 \sim \cdots \sim x_k \sim x_1$ ($k \ge 6$) is an induced cycle in $G_{\textnormal{Id}}(R_1)$. Then note that $P : (x_1, 0, \ldots,0) \sim (x_2, 0, \ldots, 0) \sim (x_3, 0, \ldots,0) \sim (x_4, 0, \ldots, 0)$ is an induced path graph of four vertices, a contradiction. It follows that either $R \cong R_1 \times R_2 \times  \cdots  \times R_n$ or  $R \cong \mathbb{Z}_3 \times R_2 \times  \cdots  \times R_n$, where $char(R_i) = 2$ for each $i$.
 
 Now let $R \cong R_1 \times R_2$. Suppose $char(R_i) \geq 4$ for some $i$. Then by Theorem \ref{idempotent_subringpath}, we get an induced subgraph isomorphic to $P_4$, a contradiction. Therefore, $char(R_i) \leq 3$. Let $char(R_i) = 3$ for each $i \in \{ 1,2\}$. By Theorem \ref{idempotent_subringpath}, we have induced paths $a_1 \sim a_2 \sim a_3$ in $G_{\textnormal{Id}}(R_1)$ and $b_1 \sim b_2 \sim b_3$ in $G_{\textnormal{Id}}(R_2)$. It follows that the graph $G_{\textnormal{Id}}(R)$ contains an induced path $P: (a_1, 0) \sim (a_2, 1) \sim (a_3, 0) \sim (a_3, 1)$ of four vertices, a contradiction. Therefore, $char(R_i) = 2$ for some $i$. Next let, $char(R_1) = 3$ with $R_1 \ncong \mathbb{Z}_3$ and $char(R_2) = 2$. Then by Theorem \ref{idempotent_subringpath}, we get an induced cycle, $x_1 \sim x_2 \sim \cdots \sim x_k (k \ge 2)$, of length at least six in $G_{\textnormal{Id}}(R_1)$. Consequently, $(x_1,0) \sim (x_2,0), \sim (x_3,0) \sim (x_4,0)$ is an induced path of $4$ vertices in  $G_{\textnormal{Id}}(R_1 \times R_2)$, a contradiction. 
 
 Conversely, assume that $char(R_i) = 2$ for each $i$, where $1 \leq i \leq n$. Let $(x_1, x_2, \ldots, x_n) \sim (y_1, y_2, \ldots, y_n) \sim (z_1, z_2, \ldots, z_n) \sim (w_1, w_2, \ldots, w_n)$ be an induced subgraph of $G_{\textnormal{Id}}(R)$ isomorphic to $P_4$. Then by Proposition \ref{idempotent_degree}, for each $u \in V(G_{\textnormal{Id}}(R_i))$, we have $deg(u) = 1$. It follows that $(x_1, x_2, \ldots, x_n) \sim (z_1, z_2, \ldots, z_n)$, which is not possible. Consequently, $G_{\textnormal{Id}}(R)$ is a cograph.
 Next, let $R \cong \mathbb{Z}_3 \times R_2 \times R_3 \times \cdots \times R_n$, where $char(R_i) = 2$ for each $i \in \{2,3,\ldots, n\}$. Let if possible, there exists an induced path $P_4 :  \ u_1 \sim u_2 \sim u_3 \sim u_4$, where $u_1 = (a_1, a_2, \ldots, a_n)$, $u_2 = (b_1, b_2, \ldots, b_n)$, $u_3 = (c_1, c_2, \ldots, c_n)$ and $u_4 = (d_1, d_2, \ldots, d_n)$. If $a_1 = 0$, then $c_1 =2 = d_1 $. Otherwise, $u_1 \sim u_3$, a contradiction. It follows that $b_1 = 1$ and so $u_2 \sim u_4$, a contradiction. Now, if $a_1 = 1$, then $c_1 = 1= d_1 $. It follows that $u_3 \nsim u_4$ in $G_{\textnormal{Id}}(R)$, a contradiction. Similarly, for $a_1 = 2$, we get $c_1 = d_1 = 0$. Since $u_2 \nsim u_4$, we must have $b_1 = 2$. It follows that $u_2 \nsim u_3$, which is not possible. Hence, $G_{\textnormal{Id}}(R)$ is a cograph.
\end{proof}

\vspace{.3cm}
\textbf{Acknowledgement:} The first author gratefully acknowledges Birla Institute of Technology and Science (BITS) Pilani, Pilani campus, India, for providing financial support.


\end{document}